\documentclass[11pt,a4paper]{article}

\usepackage{amsmath,amsthm,amsfonts}
\usepackage{graphicx}
\usepackage{color}
\usepackage{bbm}

\usepackage{hyperref}
\hypersetup{%
plainpages=false,
colorlinks,urlcolor=blue,linkcolor=blue,citecolor=blue,
pdftitle={Scheduling of energy storage},
pdfauthor={Stan Zachary, Simon Tindemans, Michael Evans, James Cruise
  and David Angeli},
pdfsubject={Scheduling of energy storage},
pdfpagemode={UseOutlines},
pdfpagelayout={OneColumn},
pdfstartview={Dubhe}
}

\newtheorem{theorem}{Theorem}[section]

\newtheorem{corollary}{Corollary}[section]

\theoremstyle{remark}

\newtheorem{remark}{Remark}[section]
\newtheorem{example}{Example}[section]

\numberwithin{equation}{section}

\renewcommand{\S}{S}
\newcommand{\ggddf}{GGDDF}
\newcommand{\ggcdf}{GGCDF}
\newcommand{\epts}[3]{\ensuremath{{#1}_s^{#2}(#3)}}
\newcommand{\eptd}[3]{\ensuremath{{#1}_d^{#2}(#3)}}

\parskip\smallskipamount
\parskip \smallskipamount

\title{Scheduling of energy storage}
\author{Stan Zachary\footnote{Heriot-Watt University, Edinburgh, UK.
    Research supported by EPSRC grant EP/I017054/1},
  Simon Tindemans\footnote{Delft University of Technology,
    Netherlands and Alan Turing Institute, London, UK},
  Michael Evans\footnote{Imperial College London, UK},\\
  James Cruise\footnote{Heriot-Watt University.  Research
    supported by EPSRC grant EP/I017054/1}
  \ and David Angeli\footnote{Imperial College London, UK and
    University of Florence, Italy}
}
\date{\today}

\begin{document}

\maketitle

\begin{abstract}
  The increasing reliance on renewable energy generation means that
  storage may well play a much greater role in the balancing of future
  electricity systems.  We show how heterogeneous stores, differing in
  capacity and rate constraints, may be optimally, or nearly
  optimally, scheduled to assist in such balancing, with the aim of
  minimising the total imbalance (unserved energy) over any given
  period of time.  It further turns out that in many cases the optimal
  policies are such that the optimal decision at each point in time is
  independent of the future evolution of the supply-demand balance in
  the system, so that these policies remain optimal in a stochastic
  environment.
\end{abstract}

\section{Introduction and model}
\label{sec:introduction}

A future much greater reliance on renewable energy means that there is
likely to be corresponding much greater need for storage in order to
keep electricity systems balanced---see~\cite{Strbacetal,NSPN}.  The
optimal operation of energy storage for such balancing may be
considered from the viewpoint of the provider
(see~\cite{SDJW,CZ,DEKM,GTL,PADS,WEITZEL2018582} and the references
therein), or from that of the system operator, who is seeking to
schedule given storage resources so as to balance the system as far as
possible.  The latter problem has only received significant attention
relatively recently, and then mostly for the problem of scheduling
initially full stores so as to cover periods of continuous energy
shortfall---see, e.g., \cite{NGECR18} for practical applications in
the context of the GB energy system, \cite{KHAN201839, Sioshansi2014,
  ZHOU201512} for dynamic programming and simulation approaches,
\cite{ETA-pscc,ETA-tosg, ETA-tops, CZ2018} for underlying theory and
\cite{Zhu2019} for a hybrid approach that uses an analytical discharge
policy and a recharging policy based on machine learning.

The present paper considers the problem of optimally scheduling
heterogeneous storage resources---characterised by different
capacities, input/output rate constraints, and round-trip
efficiencies---over extended periods of time in which there are both
periods of energy shortfall to be met from storage and periods of
energy surplus available to recharge storage.  Our main objective is
the minimisation of unserved energy demand over any given time
horizon, or, in a stochastic environment, the expectation of
this---often referred to as \emph{expected-energy-not-served} (EENS).
We work primarily in a deterministic environment.  Progress may be
made using strong Lagrangian or other closely related
techniques---see, e.g.,~\cite{Whi} and, for an application of this
approach to the present problem, see~\cite{CZ2018}.  However, in a
deterministic environment optimal solutions are typically far from
unique, and the optimal solution obtained by such techniques as above
is often such that at any point in time the optimal decision has some
dependence on the future evolution of the external supply-demand
balance process.  The approach of the present paper does not require
the use of such machinery and, for the problems considered in the
paper, it turns out that the optimal (or near optimal) policy is
generally \emph{non-anticipatory}, i.e.\ the optimal decision at each
successive time is independent of the evolution of the external
supply-demand balance subsequent to that time, so that results
continue to hold in a stochastic environment, both for the
minimisation of the unserved energy demand and for the minimisation of
any quantile of the distribution of the unserved energy demand.

We thus study a given set~$\S$ of energy stores.  At each
time~$t\in[0,\infty)$, the volume of usable energy (after accounting
for any output losses) within each store $i\in\S$ is given by $E_i(t)$,
where the latter is subject to the \emph{capacity constraint}
\begin{equation}
  \label{eq:1}
  0 \le E_i(t) \le \overline E_i. 
\end{equation}
A \emph{policy} for the use of the stores is a specification, for each
time~$t\ge0$, of the rate~$r_i(t)$, at which each store $i\in\S$
serves energy at time~$t$, where positive values of $r_i(t)$
correspond to the store discharging, and negative values of $r_i(t)$
correspond to the store charging.  Hence, in particular,
\begin{equation}
  \label{eq:2}
  E_i(t)=E_i(0)-\int_0^tr_i(u)\,du,
  \quad\text{ for all $t\ge0$.}
\end{equation}
For each store $i\in\S$, the rates~$r_i(t)$ are further required to
satisfy the power or \emph{rate constraints}
\begin{equation}
  \label{eq:3}
  -P'_i \le r_i(t) \le P_i,
  \quad\text{ for all $t\ge0$,}
\end{equation}
for given constants $P'_i\ge0$ and $P_i>0$.  Finally, each store
$i\in\S$ is assumed to have a \emph{round-trip efficiency}
$0<\eta_i\le1$, so that, at any time~$t$ such that the store is
charging (i.e.\ $r_i(t)<0$) the rate at which energy must be supplied
to the store from some external source is given by $-r_i(t)/\eta_i$
(recall that the level of energy within a store is measured as that
which it may usefully output).

The stores are used to assist in managing some given \emph{demand
  process} $(d(t), t\ge0)$, defined for all times~$t\ge0$, positive
values of which correspond to an external energy demand to be met
(perhaps partially) from the stores, and negative values of which
correspond to an external energy surplus which may be used to recharge
the stores.

In Section~\ref{sec:pure-discharge-model} we study the case in which
the demand process $(d(t),\,t\in[0,T])$ is nonnegative over some given
time interval $[0,T]$ and is to be served as far as possible over that
interval by the stores, subject to the constraint that the latter may
only discharge for $t\in[0,T]$.  In particular this may be appropriate
to the situation in which storage is used to cover continuous periods
of what would otherwise be energy shortfall (e.g.\ periods of daily
peak demand), but may readily be fully recharged between such periods.
We show that there is a policy which minimises the unserved energy
demand 
\begin{equation}
  \label{eq:4}
  \int_0^T\max\biggl(d(t) - \sum_{i\in\S}r_i(t),\,0\biggr)\,dt,
\end{equation}
and in which the rates~$r_i(t)$, $i\in\S$, at each time~$t$ depend
only on $d(t)$ and the energy in each store at time~$t$.  This policy
therefore continues to be optimal in a stochastic environment.  The
results in this section gather together---and, in considering
arbitrary time intervals, extend---results collectively obtained in
\cite{Nash1978, ETA-pscc, ETA-tosg, ETA-tops, CZ2018}, but are now
unified and presented with considerably simpler proofs, laying a
necessary foundation for subsequent sections.

In Section~\ref{sec:disch-with-cross} we continue to assume a
nonnegative demand process over some time interval $[0,T]$, but allow
that individual stores may both charge and discharge over that
interval, typically corresponding to the situation is which
\emph{cross-charging} between stores is allowed.  We show that such
cross-charging may occasionally be helpful, but give results which
identify common situations in which it is not.  In particular, we show
that cross-charging cannot be helpful when, as discussed in the
preceding paragraph, storage may be fully recharged between periods of
external energy shortfall and in which the energy shortfall during
such periods is \emph{unimodal}, increasing to a maximum and
thereafter decreasing.

Finally, in Section~\ref{sec:charging-discharging} we consider the
general case in which the demand process may be both positive and
negative, where negative values have the interpretation given above.
We study the situation in which the stores have a common round-trip
efficiency, and use earlier results to identify heuristic policies for
the (near) optimal management of the storage, and to derive conditions
under which they are truly optimal.

\section{Pure discharge model}
\label{sec:pure-discharge-model}

In this section we take the demand process $(d(t), t\ge0)$ to be
nonnegative over some time interval of interest.  Without loss of
generality---by, if necessary, redefining the demand process to be
zero outside it---we may take this time interval to consist of the
entire positive half-line, so that $d(t)\ge0$ for all $t\ge0$.  We
assume that, throughout this time interval, each store~$i\in\S$ may
only discharge, so that the rate constraint~\eqref{eq:3} is replaced
by
\begin{equation}
  \label{eq:5}
    0 \le r_i(t) \le P_i,
  \quad\text{ for all $t\ge0$,}
\end{equation}
and that
\begin{equation}
  \label{eq:6}
  \sum_{i\in\S}r_i(t) \le d(t),
  \qquad t \ge 0.
\end{equation}
The energy~$E_t(t)$ in each store~$i$ at each time~$t\ge0$ is then as
given by~\eqref{eq:2} and is a (weakly) decreasing function of~$t$
which we continue to require to satisfy~\eqref{eq:1} (though the
second inequality in~\eqref{eq:1} plays no part in the analysis of
this section).  Our objective is to choose rate processes
$(r_i(t),\,t\ge0)$ for all stores~$i\in\S$, satisfying the above
constraints, with the objective of either satisfying~\eqref{eq:6} with
equality for all~$t$ in some interval $[0,T]$ where $T$ is as large as
possible, or else that of minimising the \emph{unserved energy
  demand}, given by~\eqref{eq:4}, over any given time interval $[0,T]$
(where we allow also $T=\infty$).

Under any given policy for the use of the stores in~$\S$, the
\emph{further} capabilities of the energy contained within that set at
and subsequent to any given time~$t\ge0$ are defined by the rate
constraints~$P_i$ and by the residual stored energies
$(E_i(t),\,i\in\S)$ at time~$t$.  It is helpful to have some efficient
way of representing this \emph{residual stored-energy configuration}.
This should be sufficient to characterise the set of residual demand
processes $(d(t'),\,t'\ge t)$ which may be fully served at and
subsequent to the time~$t$.  For any such time~$t$, define the
(residual) \emph{burst-power profile} of the stored-energy
configuration at that time as the necessarily decreasing function
$s^t(u)$ of $u$ given by
\begin{equation}
  \label{eq:7}
  s^t(u) = \sum_{i\in\S\colon E_i(t)/P_i \ge u} P_i
\end{equation}
(see also Figure~\ref{fig:lgtf}).
For each store~$i\in\S$, we refer to the quantity $E_i(t)/P_i$ as the
\emph{(residual) discharge-duration} of the store~$i$ at the time~$t$.
This is the length of further time for which the store~$i$ could
supply energy if it did so at its maximum rate.  Thus the integral of
$s^t(u')$ from~$0$ to any time~$u$ is the maximum amount of further
energy which can possibly be supplied by the stores between the
times~$t$ and $t+u$.  Define also the \emph{energy-power transform}
$\epts{e}{t}{p}$ \cite{ETA-tosg} of the \emph{burst-power profile} at
the time~$t$ to be given by
\begin{equation}
  \label{eq:8}
  \epts{e}{t}{p} = \int_0^\infty \max(0, s^t(u) - p)\,du, \qquad p \ge 0.
\end{equation}
This has the interpretation that, for each~$t$ and for each $p\ge0$,
the quantity $\epts{e}{t}{p}$ would be the energy supplied above the
reference output~$p$ if all stores output at their maximum rates for
as long as possible (i.e.\ until empty) subsequent to the time~$t$
(again see Figure~\ref{fig:lgtf}).
Note that, for each time~$t$, the burst-power profile (function
of~$u$) given by $s^t(u)$ is recoverable from the energy-power
transform (function of~$p$) given by $\epts{e}{t}{p}$.
% so that the latter also defines the residual stored-energy
% configuration at time~$t$.

Observe that $\epts{e}{t}{p}$ is a (weakly) decreasing function of
both $p$ and $t$.  For each $t\ge0$, the quantity $\epts{e}{t}{0}$ is
the total energy in the stores at time~$t$, and $\epts{e}{t}{p}$ is a
convex, piecewise linear, (weakly) decreasing function of $p$ which is
zero for all $p\ge\sum_{i\in\S:E_i(t)>0}P_i$.  The quantity
$\sum_{i\in\S:E_i(t)>0}P_i$ is also the maximum rate at which demand
which may be served by the stored-energy configuration at time~$t$,
and is of course decreasing in~$t$.

For each~$T>0$ and for each~$t\in[0,T]$, define also the
\emph{energy-power transform} $\eptd{e}{t,T}{p}$ of the \emph{demand
  process} on the interval $[t,T]$ to be given by
\begin{equation}
  \label{eq:9}
  \eptd{e}{t,T}{p} = \int_t^T\max(0, d(u) - p)\,du, \qquad p \ge 0.
\end{equation}
We allow also $T=\infty$, and write $\eptd{e}{t}{p}$ for
$\eptd{e}{t,\infty}{p}$.  The quantity $\eptd{e}{t,T}{p}$ has the
interpretation that it would be the unserved energy demand over the
interval $[t,T]$ demand if power were supplied at a constant rate~$p$
during that interval.  For each~$T$ and for each $p\ge0$, the function
$\eptd{e}{t,T}{p}$ is (weakly) decreasing in~$t$ and, for each
$t\ge0$, the function $\eptd{e}{t,T}{p}$ is convex and (weakly)
decreasing in~$p$ and is zero for all $p\ge\max_{u\ge t}d(u)$.

We shall say that the (residual) stored-energy configuration
$(E_i(t), i\in\S)$ at any time~$t$ is \emph{balanced} at time~$t$ if
and only if $E_i(t)/P_i$ is constant over all stores $i\in\S$.  If the
stored energy configuration is balanced at time~$t$, then it may be
kept balanced at all subsequent times by always serving energy from
each individual store~$i\in\S$ at a rate which is proportional
to~$P_i$.  Thus, in this case and under such a policy, the stores may
fully serve any residual demand process $(d(u), u\ge t)$ such that
\begin{equation}
  \label{eq:10}
  d(u) \le \hat P, \text{ for all $u\ge t$,}
  \qquad\text{and}\qquad
  \int_t^\infty d(u)\,du \le \hat E(t),
\end{equation}
where $\hat P=\sum_{i\in\S}P_i$ and $\hat E(t)=\sum_{i\in\S}E_i(t)$.
Since, under \emph{any} policy, these conditions are clearly also
necessary in order that the stores, balanced at time~$t$ as above, may
fully serve a given residual demand process subsequent to time~$t$, it
follows that the above policy for the subsequent use of a balanced
energy configuration is optimal in terms of its ability to satisfy any
requested demand.  Indeed the balanced store configuration at time~$t$
has the same subsequent energy-serving capability as a single store
with the same total energy content $\hat E(t)$ and a maximum output
rate of $\hat P$.

When the residual stored-energy configuration $(E_i(t), i\in\S)$ at
time~$t$ is balanced as above, the corresponding energy-power
transform $\epts{e}{t}{p}$ decreases \emph{linearly} in $p$ from
$\hat E(t)$ when $p=0$ to zero when $p=\hat P$ (and is zero
thereafter).  The conditions~\eqref{eq:10} on the residual demand
process $(d(u), u\ge t)$ are equivalent to
$\eptd{e}{t}{0}\le\hat E(t)$ and $\eptd{e}{t}{\hat P}=0$, where
$\eptd{e}{t}{p}$ is the corresponding energy-power transform of that
process on $[0,\infty)$.  Since the latter is convex, it follows that
the residual demand at and subsequent to time~$t$ may be completely
served if and only if $\eptd{e}{t}{p}\le\epts{e}{t}{p}$ for all
$p\ge0$, and is then served by keeping the residual stored-energy
configuration balanced subsequent to time~$t$.  It is a consequence of
Theorem~\ref{theorem:1} below that, under a suitable policy for the
use of the stores, this result extends to general stored-energy
configurations. 

Suppose that, under a given policy for the use of the set of
stores~$\S$, the total rate at which energy is to be served at each
time~$t\ge0$ is given by $\bar d(t)\le d(t)$.  We shall say that such
a policy is a \emph{greatest-discharge-duration-first} policy if, at
each successive time~$t$, the rates at which the individual stores
$i\in\S$ serve energy is given by prioritising the use of the stores
in descending order of their discharge-durations $E_i(t)/P_i$.  More
exactly, under this policy at each time~$t$ the stores are grouped
according to their current discharge-durations~$E_i(t)/P_i$ (so that
stores belong to the same group if and only if their
discharge-durations are equal); sufficient groups are then selected in
descending order of their discharge-durations such that, using the
stores within each group at their maximum rates (i.e.\ each store~$i$
within a selected group serves energy at a rate~$P_i$), the total rate
at which energy is served is the required~$\bar d(t)$; however, in
order to meet precisely the total rate~$\bar d(t)$, each store in the
\emph{last} group thus selected may only require to serve energy at
some fractional rate~$\lambda P_i$ for some constant~$\lambda$ such
that $0<\lambda\le1$.  (See Figure~\ref{fig:lgtf}.)  As time~$t$
increases, any such greatest-discharge-duration-first policy gradually
equalises over stores the discharge-durations~$E_i(t)/P_i$, thus
pushing the residual stored energy configuration towards a balanced
state as defined above.  Additionally, under this policy, once the
discharge-durations within any set of stores have become equal they
remain equal for all subsequent times.  Thus, when groups of stores as
defined above coalesce they remain coalesced, and further the (weak)
ordering of stores by their discharge-durations remains unchanged over
time.

\begin{figure}[!ht]
  \centering
  \includegraphics[scale=0.6]{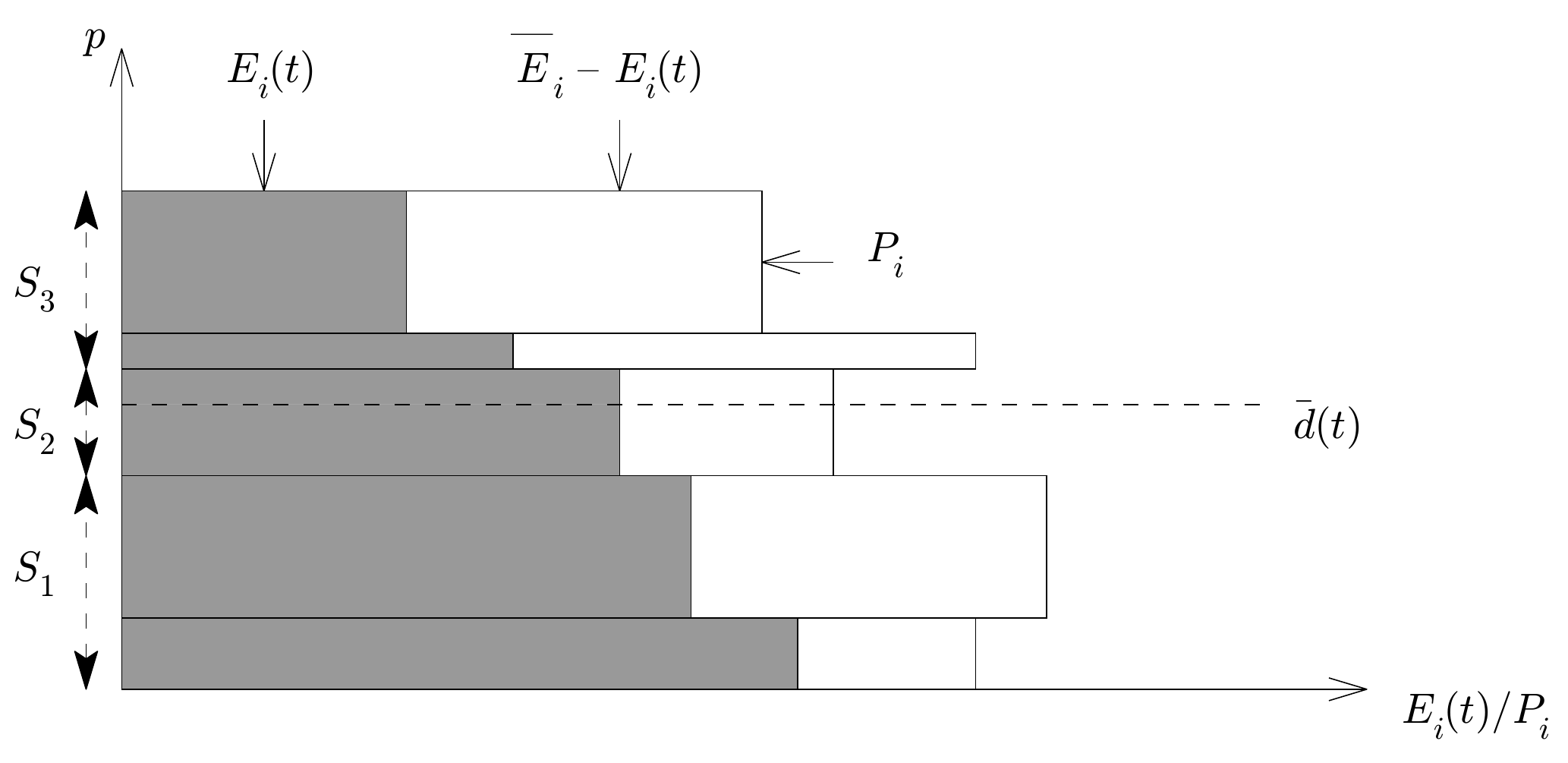}
  \caption{Greatest-discharge-duration-first policy for five stores.
    The shaded areas correspond to the residual energies $E_i(t)$
    within the stores at time~$t$.  Stores are selected in descending
    order of $E_i(t)/P_i$.  (Thus the total shaded area is the area
    under the burst-power profile at time~$t$, while the shaded area
    above any level~$p$ defines the energy-power transform
    $\epts{e}{p}{t}$ at time~$t$.)  When $\bar d(t)$ is the energy to
    be served, $\S_1$, $S_2$ and $\S_3$ are the sets of stores which
    are fully, partially, or not utilised at time~$t$.}
  \label{fig:lgtf}
\end{figure}

We shall say that a policy is \emph{greedy} if, at each successive
time~$t>0$, it serves as much as possible of the demand~$d(t)$ at that
time, i.e.\ if, under this policy,
\begin{displaymath}
  \sum_{i\in\S}r_i(t) = \min\left(d(t),\,\sum_{i\in\S: E_i(t)>0}P_i\right),
  \quad\text{for all $t\ge0$.}
\end{displaymath}
Note that there is a unique \emph{greedy
  greatest-discharge-duration-first} (\ggddf) policy.  This policy was
independently proposed in \cite{Nash1978} in the context of water
management, and in \cite{ETA-pscc, CZ2018} for the current context of
energy storage.  A discrete-time algorithm implementing this policy
for a piecewise constant demand process is given in
\cite{ETA-tops}. Finally, we note that \cite{Zhu2019} introduced a
closely related discharge policy for the class of \emph{constant}
demand signals ($d(t)=d$, for $t \in [0,T]$ and $d(t)=0$ for $t > T$),
where $d$ (but not $T$) is known to the dispatcher in advance. For
this limited class of signals, the proposed policy is optimal in the
sense that it minimises unserved energy demand. It will be shown below
that the \ggddf\ policy is optimal in the same sense, but for a much
broader class of demand signals that includes constant signals as a
special case.

Suppose now that it may not be possible to serve the entire demand
process $(d(t),\,t\ge0)$ and that our objective is the minimisation of
the total unserved energy demand~\eqref{eq:4} over some time interval
$[0,T]$.  Theorem~\ref{theorem:1} below is central to the rest of the
paper.  It gathers together and provides a unified, economical and
accessible treatment of results previously obtained by various
authors.  In particular the optimality of the \ggddf\ policy is
established in \cite{Nash1978, ETA-pscc, CZ2018},
while the implied necessary and sufficient condition for this policy
to be able to serve a given demand process is derived by
\cite{ETA-tosg, CZ2018}, and the expression for
the unserved energy demand under this policy is given by
\cite{ETA-tops}.   In considering arbitrary time
intervals, the theorem also provides a very modest extension of these
results. 

\begin{theorem}
  \label{theorem:1}
  For the given demand function $(d(t),\,t\ge0)$ and the given initial
  energy configuration $(E_i(0),\,i\in\S)$, the total unserved energy
  demand~(\ref{eq:4}) over any time interval $[0,T]$ is minimised by
  the \ggddf\ policy, and this minimum is given by
  \begin{equation}
    \label{eq:11}
%    \max\left[0,\,
      \max_{p\ge0}\bigl(\eptd{e}{0,T}{p} - \epts{e}{0}{p}\bigr).
%    \right].
  \end{equation}
\end{theorem}

\begin{proof}
  We show first that the use of the \ggddf\ policy results in unserved
  energy as given by~\eqref{eq:11}.  For each~$t\ge0$, let
  $\bar d(t) \le d(t)$ be the demand actually served at time~$t$ under
  the \ggddf\ policy.  Similarly, for each such~$t$, let $E_i(t)$ be
  the energy remaining in each store~$i\in\S$ at time~$t$ under the
  \ggddf\ policy and let $\epts{e}{t}{p}$, $p\ge0$, be the
  corresponding energy-power transform.  Suppose that, at time~$t$,
  stores in $\S$ are ranked in descending order of their
  discharge-durations $E_i(t)/P_i$ (as is required for the
  implementation of the \ggddf\ policy).  Let $\pi(t)$ be the set of
  values of~$p$ such that $p = \sum_{i=1}^j P_i$ 
  for some~$j$ such that $E_{j+1}(t)/P_{j+1} < E_j(t)/P_j$, and
  include also in the set~$\pi(t)$ the values $p=0$ and (if no store is
  empty) $p=\sum_{i\in\S}P_i$.  (See the example of
  Figure~\ref{fig:epts} below.)  It follows from the earlier
  observation that once, under the \ggddf\ policy, the
  discharge-durations of any two stores have become equal they remain
  equal thereafter, that
  \begin{equation}
    \label{eq:12}
    \pi(u) \supseteq \pi(t),
    \quad \text{ for all $0 \le u \le t$.}
  \end{equation}
  Further, from~\eqref{eq:8} and from the definition of the \ggddf\
  policy, it is readily checked that, under this policy, at each
  time~$t$ and for all $p\ge0$, the derivative with respect to~$t$ of
  the energy-power transform $\epts{e}{t}{p}$ of the residual energy
  configuration satisfies
  \begin{equation}
    \label{eq:27}
    \frac{d}{dt}(\epts{e}{t}{p}) \le \min(0, p - \bar d(t)),
  \end{equation}
  with equality for all~$p\in\pi(t)$.  The latter equality is easily
  seen, and the general result~(\ref{eq:27}) follows from the
  observation that the above derivative varies linearly in~$p$ between
  successive points of the set $\pi(t)$ (again see the example of
  Figure~\ref{fig:epts}), while the function $\min(0,p - \bar d(t))$
  is concave in~$p$.  (Indeed we have equality in~(\ref{eq:27}) for
  all values of~$p$ other than for those lying within that interval
  which is bounded by two consecutive points of the set~$\pi(t)$ and
  within which $\bar d(t)$ lies.)  Similarly, from~\eqref{eq:9}, at
  each time~$t$ and for all $p\ge0$, the derivative with respect
  to~$t$ of the energy-power transform $\eptd{e}{t,T}{p}$ of the
  residual demand is given by
  \begin{equation}
    \label{eq:28}
    \frac{d}{dt}(\epts{e}{t}{p}) = \min(0, p - d(t)).
  \end{equation}
  Thus, from~(\ref{eq:27}) and~(\ref{eq:28}), under the \ggddf\
  policy, at each time~$t$ and for all $p\ge0$,
  \begin{equation}
    \label{eq:13}
    \frac{d}{dt}(\epts{e}{t}{p} - \eptd{e}{t,T}{p}) \le 
    \begin{cases}
      d(t) - \bar d(t), &  \qquad p < \bar d(t),\\
      d(t) - p, &  \qquad \bar d(t) \le p \le d(t),\\
      0, & \qquad p > d(t),
    \end{cases}
  \end{equation}
  with equality for all $p\in\pi(t)$.
  It now follows from~\eqref{eq:13} that, under the \ggddf\ policy and
  for all $p\ge0$, the unserved energy demand over the interval
  $[0,T]$ is given by
  \begin{align}
    \int_0^T(d(t) - \bar d(t))\,dt 
    & \ge \epts{e}{T}{p} 
      - \epts{e}{0}{p} + \eptd{e}{0,T}{p} \label{eq:15}\\
    & \ge \eptd{e}{0,T}{p} - \epts{e}{0}{p}, \label{eq:16}
  \end{align}
  where \eqref{eq:16} follows since necessarily
  $\epts{ e}{T}{p} \ge 0$.  It follows that the unserved energy demand
  over the time interval $[0,T]$ is greater than or equal to
  % $\max_{p\ge0}\left(\eptd{e}{0,T}{p}-\epts{e}{0}{p}\right)$.
  the expression given by (\ref{eq:11}).  To prove equality, define
  $\hat p = \min\{p\ge0 \colon \epts{ e}{T}{p} = 0\}$.  
  Note that $\hat p$ necessarily exists and that
  $\hat p \in \pi(T)$; the latter follows, for example, from the linearity
  of $\epts{e}{T}{p}$ between adjacent points of $\pi(T)$.  Observe
  that $\hat p \le \bar d(t)$ for all $t\in[0,T]$ such that
  $\bar d(t) < d(t)$ (i.e.\ under the \ggddf\ policy there is unserved
  demand at time~$t$); this follows since, for any such~$t$,
  necessarily $\epts{ e}{t}{\bar d(t)}=0$ and so also
  $\epts{ e}{T}{\bar d(t)}=0$.  Observe also that, by~\eqref{eq:12},
  $\hat p \in \pi(t)$ for all $t\in[0,T]$.  It follows from the above
  two observations, and by using~\eqref{eq:13}, that~\eqref{eq:15}
  holds with equality for $p=\hat p$.  Since also
  $\epts{e}{T}{\hat p}=0$, the relation~\eqref{eq:16} also holds with
  equality for $p = \hat p$.  Hence the expression~(\ref{eq:11}) also
  provides an upper bound on the unserved energy demand over the
  interval $[0,T]$ under the \ggddf\ policy.  
  (Note that, combining this with the
  earlier lower bound, it follows (i) that the quantity 
  $\eptd{e}{0,T}{p} - \epts{e}{0}{p}$ is maximised 
  for $p=\hat{p}$, and (ii) that when the demand function 
  can be fully served, we have $\eptd{e}{0,T}{\hat{p}} = \epts{e}{0}{\hat{p}}$.)  
  It now follows that, under the \ggddf\ policy,
  the unserved energy demand is as given by~\eqref{eq:11}.

  We now show that any nonnegative demand process $(d(t),\,t\ge0)$
  which may be completely served over the interval $[0,T]$ by
  \emph{some} policy, may be completely served over that interval by
  the use of the \ggddf\ policy.  To do this, it is sufficient to show
  that the condition
  \begin{equation}
    \label{eq:17}
    \epts{e}{0}{p} \ge \eptd{e}{0,T}{p}
    \qquad\text{for all $p \ge 0$}.
  \end{equation}
  is necessary (as well as sufficient) for the demand process
  $(d(t),\,t\ge0)$ to be capable of being completely served over the
  interval $[0,T]$.  For any $p\ge0$, let $\tau_p$ be the of the set
  of times~$t$ within the interval $[0,T]$ such that $d(t)\ge p$.  For
  the demand process $(d(t),\,t\ge0)$ to be completely served over the
  interval $[0,T]$, it is necessary that
  \begin{equation}
    \label{eq:18}
    \int_0^{m(\tau_p)} s^0(u)\,du
    \ge
    \int_{\tau_p} d(u)\,du,
  \end{equation}
  where $m(\tau_p)$ is the total length of the set of times $\tau_p$.
  That this is so follows since, from~\eqref{eq:7}, the integral on
  the left side of~\eqref{eq:18} is the maximum amount of energy which
  is capable of being served within any set of times of total length
  $m(\tau_p)$.  The relation~\eqref{eq:18} in turn implies that
  \begin{displaymath}
     \int_0^{m(\tau_p)} \max(0,s^0(u) - p)\,du
    \ge
    \int_{\tau_p} \max(0,d(u) - p)\,du,
  \end{displaymath}
  since $d(u)\ge p$ on the set~$\tau_p$.  However, this is simply the
  condition~\eqref{eq:17}.

  Finally, to complete the proof of the theorem, suppose that the
  nonnegative demand process $(d(t),\,t\ge0)$ is not necessarily
  completely served over the interval $[0,T]$ by any policy.  Consider
  any policy which minimises the unserved energy demand over the
  interval $[0,T]$ and let $(\hat d(t),\,t\ge0)$, with
  $\hat d(t) \le d(t)$ for all~$t$ and $d(t)=0$ for $t>T$, be the
  process of such demand as \emph{is} served over $[0,T]$ under that
  policy.  Then, by the result of the preceding paragraph (with
  $\hat d(t)$ replacing $d(t)$ for all~$t$), the process
  $(\hat d(t),\,t\ge0)$ may also be completely served over the
  interval $[0,T]$ by the use of the \ggddf\ policy.  The \ggddf\
  policy therefore also minimises the unserved energy demand over that
  interval associated with the original process $(d(t),\,t\ge0)$.
\end{proof}

Theorem~\ref{theorem:1} has the following immediate corollary, which
is fundamental in establishing the energy-power transform of a
stored-energy configuration as containing all the information as to
which future demand processes may be completely served.

\begin{corollary}
  \label{corollary:1}
  Any given demand process $(d(t), t\ge0)$ may be completely served
  over any interval $[0,T]$ (i.e.\ the unserved energy
  demand~(\ref{eq:4}) is zero) by a given energy configuration with
  (initial) energy-power transform $(\epts{e}{0}{p},\,p\ge0)$ if and
  only if $\epts{e}{0}{p} \ge \eptd{e}{0,T}{p}$ for all $p \ge 0$.
  Under this condition the demand process is completely served by the
  use of the \ggddf\ policy \cite{ETA-tosg}.
\end{corollary}

\begin{remark}\label{remark:1}
  It follows from Corollary~\ref{corollary:1} and the properties of
  the function $\epts{e}{0}{p}$ noted earlier that, for a given total
  volume of stored energy $\sum_{i\in\S}E_i(0)$ at time~$0$, the set
  of future demand processes which may be completely served is
  maximised when the stored-energy configuration at time~$0$ is
  \emph{balanced} as defined above, so that the corresponding
  energy-power transform $\epts{e}{0}{p}$ decreases linearly in~$p$.
\end{remark}

The \ggddf\ policy has the important property of being
\emph{non-anticipatory}---as defined in
Section~\ref{sec:introduction}.  It follows that the \ggddf\ policy
remains feasible within a stochastic environment, i.e.\ when, at each
successive time~$t$, the demand function $d(t')$ is known for times
$t'\le t$, but not necessarily for $t'>t$.  Since, by
Theorem~\ref{theorem:1}. the \ggddf\ policy thus minimises
unserved energy demand for all possible evolutions of the
demand function, we have the following
further corollary to that theorem.

\begin{corollary}
  \label{corollary:2}
  Suppose that (in a stochastic environment) the objective for the
  optimal serving of energy is the minimisation of the expectation of
  the unserved energy demand~(\ref{eq:4}), or the
  minimisation of any quantile of the distribution of
  the unserved energy demand.  Then this objective is
  achieved by the use of the unique \ggddf\ policy \cite{ETA-tops,
    CZ2018}.
\end{corollary}

\begin{corollary}
  \label{corollary:2b}
  Suppose that a possibly stochastic demand signal $(d(t),t\ge 0)$
  cannot be completely served. In this case, consideration of the
  truncated signals on $t \in [0,\tilde{T}]$, for $\tilde{T}\ge0$
  shows that the \ggddf\ policy maximises the time until the storage
  fleet is first unable to serve all demand \cite{ETA-pscc}.
\end{corollary}

\begin{example}\label{ex:ggddf}
  We illustrate various features of the \ggddf\ and other policies
  with a simple example, which is adapted to present needs from one
  given by~\cite{CZ2018} and which is reasonably realistic in its
  dimensioning.  Consider a time interval $[0,4]$ and a demand process
  $(d(t),\,t\in[0,4])$ given by
  \begin{displaymath}
    d(t) =
    \begin{cases}
      200, & \quad 0 \le t \le 2,\\
      500, & \quad 2 < t \le 3,\\
      100, & \quad 3 < t \le 4.
    \end{cases}
  \end{displaymath}
  Consider also $5$ stores, initially full, each with the same rate
  constraint $P_i=100$ and with values of $E_i$ given by $100$, $150$,
  $200$, $200$, $250$ for $i = 1,\dots,5$.  These stores are to be
  used to serve as much as possible of the above demand process.  If
  time is measured in hours, power in MW (and so energy in MWh), then
  this example might correspond to a modest level of shortfall over a
  four-hour period in a country such as Great Britain, with the stores
  corresponding to four moderately large batteries.  It is readily
  verified that the \ggddf\ policy empties all five stores over the
  time period~$[0,3]$, serving all the demand during that time period,
  and none of the demand during the remaining time period $[3,4]$,
  thereby leaving (minimised) unserved energy demand of $100$.  This
  is as predicted by Theorem~\ref{theorem:1}---see
  Figure~\ref{fig:epts}, which plots the energy-power transforms
  $\epts{e}{0}{p}$ and $\eptd{e}{0}{p}$.  Various other policies also
  empty all the stores and hence minimise unserved energy demand.  One
  such is the \ggddf\ policy applied to the time-reversed demand
  process, which serves all the demand except that occurring during
  the period $[0,0.5]$ (during which none is served).  A further such
  policy is that which uses the greatest-discharge-duration-first
  policy, not greedily, but rather to serve all demand in excess of a
  given level~$k$.  For $k=25$ this policy just serves all the demand
  in excess of that level, again emptying the stores.  However,
  neither of the above two policies, viewed as algorithms, could be
  implemented in a stochastic environment, as in each case the
  decisions to be made at each successive time require a knowledge of
  the demand process over the entire time period $[0,4]$.
  
  \begin{figure}[!ht]
    \centering
    \vspace{-4ex}
    \includegraphics[scale=0.7]{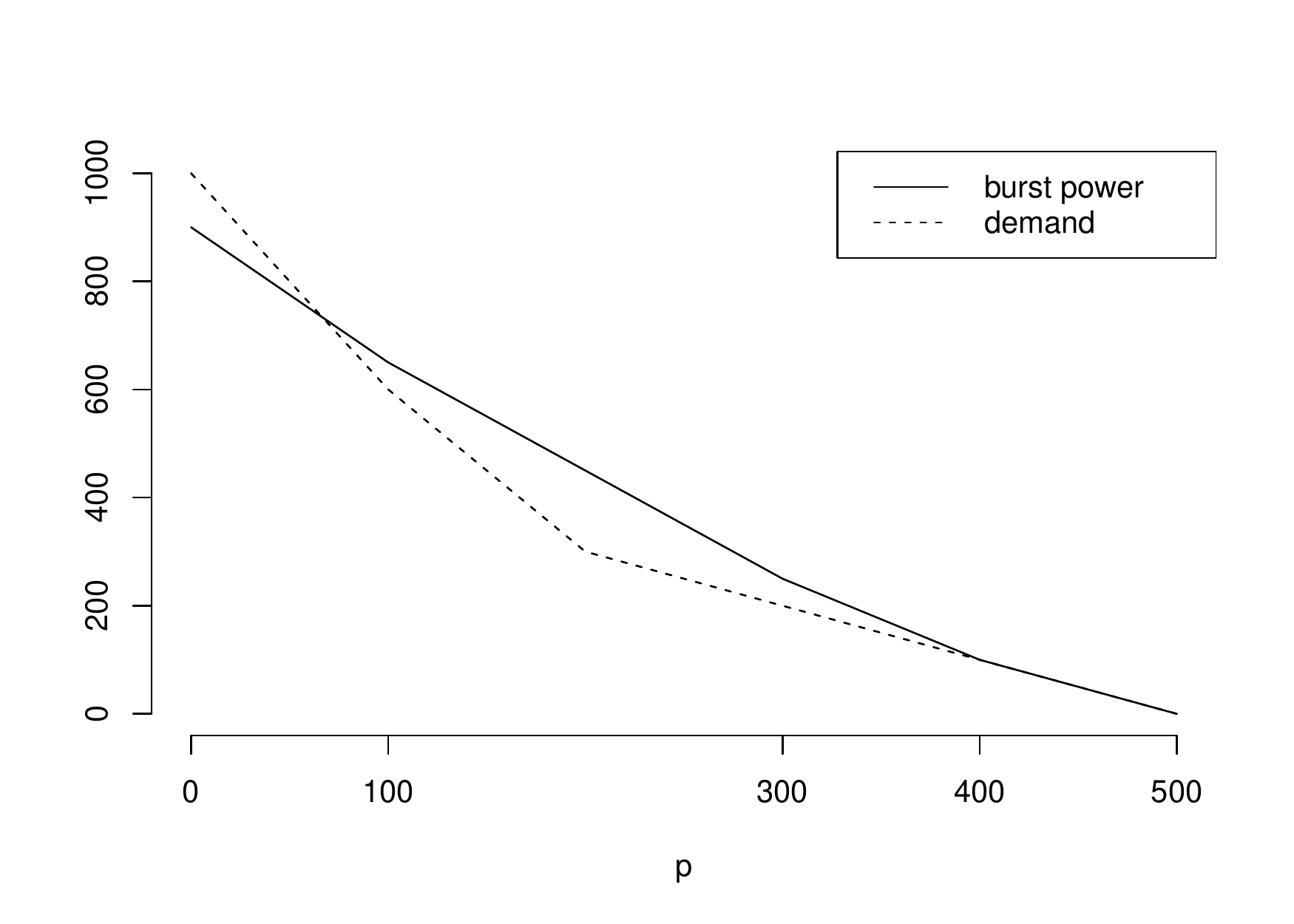}
    \vspace{-2ex}
    \caption{Example~\ref{ex:ggddf}: \emph{energy-power transforms}
      $\epts{e}{0}{p}$ and $\eptd{e}{0}{p}$ at time~$0$ of the
      burst-power profile (of the initial stored-energy configuration)
      and of the demand process $(d(t),\,t\ge0)$.  The set $\pi(0)$
      defined in the proof of Theorem~\ref{theorem:1} is given by
      $\{0,100,300,400,500\}$ and is indicated on the horizontal axis.
      Both $\epts{e}{0}{p}$ and $\eptd{e}{0}{p}$ are convex decreasing
      functions of $p$ and the function $\epts{e}{0}{p}$ is linear
      between successive points of the set $\pi(0)$.  The value of the
      expression~(\ref{eq:11}) is equal to 100
      (the maximum being attained at $p=0$), which, from
      Theorem~\ref{theorem:1} is the minimum unserved energy demand
      over all policies.}
    \label{fig:epts}
  \end{figure}

  Finally, the heuristic greedy policy studied in~\cite{ESDT} arranges
  the stores in some order and completely prioritises the use of
  earlier stores (with respect to that order) over later ones.  It may
  be checked that, with respect to the arrangement of the stores in
  either ascending or descending order of capacity, the suggested
  policy fails to empty the stores and hence fails to minimise
  unserved energy demand.
\end{example}

\section{Cross-charging of stores}
\label{sec:disch-with-cross}

We continue to study the situation in which the demand process
$(d(t),\,t\ge0)$ is nonnegative over some time interval, which we
again take, without loss of generality, to be the entire positive
half-line, so that $d(t)\ge0$ for all $t\ge0$.  We now allow that
stores may be charging as well as discharging, so that, for each
$i\in\S$, the rates $r_i(t)$ satisfy the rate
constraints~\eqref{eq:3}.  However, we require also that the net
energy supplied by the stores in $\S$ is positive for all $t\ge0$ and
is used to satisfy, again as far as possible in some suitable sense,
the demand process $(d(t),\,t\ge0)$.  Hence we require
\begin{equation}
  \label{eq:19}
    0 \le \sum_{i\in\S:r_i(t)\ge0}r_i(t) +
    \sum_{i\in\S:r_i(t)<0}r_i(t)/\eta_i \le d(t),
  \qquad t \ge 0,
\end{equation}
where, as previously defined, $0<\eta_i\le1$ is the round-trip
efficiency of each store $i\in\S$.  This corresponds to the situation
in which stores may supply energy to each other---which we refer to as
\emph{cross-charging}---but in which no external energy is available
for the charging of stores.  Our main aim in this section is to show
that, while such cross-charging may often assist in serving an
external demand, it is also possible to identify circumstances, of
considerable importance in practical applications, in which it does
not.  We give first a simple example in which cross-charging is
helpful.
\begin{example}
  \label{example:cc}
  Consider two stores with capacity and rate constraints~\eqref{eq:1}
  and~\eqref{eq:3} given by $(\overline E_1,P_1)=(2,2)$,
  $(\overline E_2,P_2)=(4,1)$ and $P_i=P'_i$ for $i=1,2$, and assume
  further that each store~$i$ has a round-trip efficiency~$\eta_i=1$.
  We take the demand process $(d(t),\,t\ge0)$ to be given by $d(t)=3$
  for $t\in[0,1]$ and $t\in[3,4]$ and $d(t)=0$ otherwise.  Finally, we
  assume that the two stores are full at time~$0$.  Then it is
  straightforward to see the only way in which the demand signal can
  be completely served for all $t\ge0$ is to fully empty store~$1$ and
  use one unit of energy from store~$2$ during the time period
  $[0,1]$, to fully recharge the store~$1$ from store~$2$ during the
  time period $[1,3]$, and then to fully discharge both stores during
  the time period $[3,4]$.
\end{example}

In the above example there are initial and final periods of high
demand, requiring service from both stores, separated by a period of
low demand during which the low capacity store may be recharged from
the high capacity store.  This is the typical situation in which
cross-charging may be useful.  However, for round-trip efficiencies
which are less than one, such cross-charging is inherently wasteful of
energy.  The following theorem is now fundamental.

\begin{theorem}
  \label{theorem:2}
  Suppose that the demand process $(d(t),t\ge0)$ is (weakly) decreasing
  for all~$t\ge0$, and is such that it may be served, possibly with
  cross-charging and subject to the conditions~(\ref{eq:19}), by the
  given stored-energy configuration.  Then it may also be served
  without cross-charging (i.e.\ with $r_i(t)\ge0$ for all $i$ and for
  all $t$) by the use of the \ggddf\ policy.
\end{theorem}

\begin{proof}
  We prove the implication of the theorem by proving the
  contrapositive: if the \ggddf\ policy (which does not permit
  cross-charging) cannot serve the demand process $d(t)$, then neither
  can any other policy, including those that make use of
  cross-charging.
  
  Clearly, if $d(0)>\sum_{i \in \S: E_i(0)>0} P_i$ (so that there is
  insufficient power at time $t=0$) or if
  $\int_0^{\infty} d(t)\mathrm{d}t > \sum_{i \in \S} E_i(0)$ (i.e.\
  the total energy in all the stores is insufficient to meet the
  demand process), then no policy is able to serve the demand process.
  Hence, we consider only the remaining cases where the \ggddf\ policy
  fails to fully serve the demand.  In these cases, there must be a
  first failure time $t=t^*$, characterised by the condition
  \begin{equation}
    \label{eq:29}
    \sum_{i \in \S^*} P_i < d(t^*),
  \end{equation}
  where $\S^* = \{i \in \S: E_i(t^*)>0\}$ is the set of stores that are
  not empty at time~$t^*$.
  
  The \ggddf\ policy has the property that it preserves through time
  the ordering of discharge-durations $E_i(t)/P_i$, except for
  equalisations.  Since the stores in $\S^*$ have the highest values
  of $E_i(t^*)/P_i$ at time~$t^*$ (the other stores being empty at
  that time), this must also have been the case for all times $t<t^*$.
  Since, further, the demand process $(d(t),t\ge0)$ is (weakly)
  decreasing, it now follows from~(\ref{eq:29}) that,
  under the \ggddf\ policy, all the stores in the set~$\S^*$ have
  served energy at their maximum rates (i.e.\ $r_i(t) = P_i$ for all
  $i\in\S^*$) for all times~$t\in[0,t^*]$.  Since also the stores
  $i\notin\S^*$ are empty at time~$t^*$, it now follows that at
  time~$t^*$ the \ggddf\ policy has extracted the maximum possible
  amount of energy from the entire set of stores~$\S$ over
  $t\in[0,t^*]$.  Hence, under any policy (with or without
  cross-charging) which succeeded in serving all demand over the time
  interval $[0,t^*]$, the stores $i\notin\S^*$ would also be empty at
  time~$t^*$, so that that policy would also fail at that time. 
 \end{proof}

Theorem~\ref{theorem:2} has the following companion result.

\begin{theorem}
  \label{theorem:3}
  Suppose that the demand function $(d(t),t\in[0,T])$ is \emph{(weakly)
    increasing} on $[0,T]$, and that all stores are full at time~$0$.
  Then if $(d(t),t\in[0,T])$ can be served, it can be served without
  cross-charging (i.e.\ with $r_i(t)\ge0$ for all $i$ and for all
  $t$), and with at least as much energy remaining in each store at
  the final time~$T$.
\end{theorem}

\begin{proof}
  We first consider the case where every store $i\in\S$ has a
  round-trip efficiency $\eta_i=1$, so that there is no loss of energy
  in cross-charging.  Here the result can be deduced from
  Theorem~\ref{theorem:2} by an argument involving time and space
  reversal.  Consider any policy for the use of the stores, possibly
  involving cross-charging, which serves the given demand process
  $(d(t),t\in[0,T])$.  For each store $i\in\S$, let $E_i(t)$ be the
  corresponding level of store~$i$ at each time~$t$; define a new
  sequence of storage levels $(E^*_i(t),t\in[0,T])$ on the
  interval~$[0,T]$ for the store~$i\in\S$ by
  \begin{equation}
    \label{eq:21}
    E^*_i(t) = \overline E_i - E_i(T-t),
    \qquad t \in [0,T].
  \end{equation}
  % Note that $E^*_i(0) = \overline E_i - E_i(T)$ and
  % $E^*_i(T) = 0$.
  The set of such sequences over all $i\in\S$ corresponds to the use
  of the stores, with the same input and output rate constraints and
  again with no loss of energy in cross-charging, to serve fully a
  demand process $(d^*(t),t\in[0,T])$ given by $d^*(t) = d(T-t)$ for
  all $t$, with the initial level of each store~$i$ being given
  (from~\eqref{eq:21}) by
  $E^*_i(0) = \overline E_i - E_i(T) \le \overline E_i$ and with the
  final level of every store being given by~$E^*_i(T) = 0$ (again
  from~\eqref{eq:21} since, by the hypothesis of the theorem,
  $E_i(0)=\overline E_i$ for all~$i$).  Further, this pattern of use
  of the stores involves cross-charging if and only if the use of the
  original set of sequences $(E_i(t),t\in[0,T])$, for $i\in\S$, to
  serve the original demand process $(d(t),t\in[0,T])$ similarly
  involves cross-charging.  Since the original demand process is
  increasing, the demand process $(d^*(t),t\in[0,T])$ is decreasing
  and so, by Theorem~\ref{theorem:2}, it may also be served fully,
  without cross-charging, by a modified set of sequences of store
  levels $(\hat E^*_i(t),t\in[0,T])$, $i\in\S$, with
  $\hat E^*_i(0) = E^*_i(0)$ and $\hat E^*_i(T) = E^*_i(T) = 0$ for
  all~$i$ (this last result following since the sets of processes
  $(E^*_i(t),t\in[0,T])$, $i\in\S$, and $(\hat E^*_i(t),t\in[0,T])$,
  $i\in\S$, both start at the same set of levels, serve the same total
  volume of energy over the period $[0,T]$, and, since the former set
  of processes fully empties the set of stores, so also must the
  latter).  Finally, transforming back in time and space, it follows
  that set of sequences of store levels $(\hat E_i(t),t\in[0,T])$,
  $i\in\S$, given by
  \begin{displaymath}
    \hat E_i(t) = \overline E_i - \hat E^*_i(T-t),
    \qquad t \in [0,T],
  \end{displaymath}
  fully serves the original demand process without cross-charging, and
  that $\hat E_i(0) = \overline E_i$ and $\hat E_i(T) = E_i(T)$ for
  all $i\in\S$, as required, and indeed so that in this case the
  modified process, with cross-charging eliminated, leaves exactly the
  same volume of energy in each store at the final time~$T$.
  
  We now consider the general case $\eta_i\le1$ for all $i\in\S$.
  Again consider any policy for the use of the stores, possibly
  involving cross-charging, which serves the given demand process
  $(d(t),t\in[0,T])$.  In particular the rates~$r_i(t)$ associated
  with this policy satisfy~(\ref{eq:19}) with the second inequality in
  that expression replaced with equality.  As usual, we denote by
  $E_i(t)$ the energy level in each store $i\in\S$ at each
  time~$t\ge0$.  Consider also a modified model in which the stores
  and demand process remain the same, except only that the round-trip
  efficiencies~$\eta_i$ are all replaced by one.  It is clear that we
  may choose a policy (possibly including cross-charging), i.e.\ a set
  of rate functions $(\hat r_i(t),t\ge0),\,i\in\S$, for this modified
  model in which again all demand is served, i.e.\
  \begin{equation}
    \label{eq:30}
    \sum_{i\in\S}\hat r_i(t) = d(t), \qquad t\ge0,
  \end{equation}
  and in which, for each time~$t\ge0$, the corresponding energy levels
  $\hat E_i(t)$ in the stores satisfy
  \begin{equation}
    \label{eq:25}
    \hat E_i(t) \ge E_i(t), \qquad i\in\S.
  \end{equation}
  To see this observe that we may, inductively over time, choose the
  rates $\hat r_i(t)$, and hence the store levels $\hat E_i(t)$ as
  follows: at each time~$t$, for those $i$ such that $r_i(t) < 0$, set
  $\hat r_i(t) = r_i(t)$ unless $E_i(t) = \overline E_i$ (store~$i$ is
  full) in which case set $\hat r_i(t) = 0$; similarly, at each
  time~$t$, for those~$i$ such that $r_i(t) \ge 0$, choose
  $0 \le \hat r_i(t) \le r_i(t)$ and such that equation~(\ref{eq:30})
  is satisfied for the modified model.  Arguing inductively, it is
  easy to see that, at each successive point in time, this modified
  model preserves all rate and capacity constraints together with the
  relation~(\ref{eq:25}).  (Informally we might think of the modified
  model as corresponding to the idea that the same external demand is
  notionally served at the same rates from the same stores as in the
  unmodified model, while the perfect round-trip efficiencies of the
  stores in the modified model enable cross-charging to be used to
  ensure that, at all times, the relation~(\ref{eq:30}) is satisfied
  and all store levels are greater than or equal to their levels in
  the unmodified model.)  By the result of the theorem already
  proved for the case of perfect round trip efficiency, the above
  policy may now be further modified so as to serve the same demand
  process while eliminating cross-charging and leaving the final
  volume of energy (at time~$T$) in each store $i\in\S$ unchanged by
  this elimination.  Since no cross-charging is now taking place, this
  further modified policy now also serves the demand in the original
  system with $\eta_i\le1$ for all~$i$, again with the same final
  volume of energy in each store at time~$T$, and this volume is
  therefore greater than or equal to that which was present in the
  original model when cross-charging was used.
\end{proof}

The need, in Theorem~\ref{theorem:3}, for some condition such as the
requirement that all stores are full at time~$0$ is shown by
Example~\ref{example:cc2} below.

\begin{example}
  \label{example:cc2}
  Consider two stores with capacity and rate constraints~\eqref{eq:1}
  and~\eqref{eq:3} given by $(\overline E_1,P_1)=(2,1)$,
  $(\overline E_2,P_2)=(1,1)$, with $P_i=P'_i$ for $i=1,2$, and with
  round-trip efficiencies~$\eta_i=1$ for $i=1,2$.  Consider a time
  horizon $T=2$ and let the demand process $(d(t),\,t\in[0,2])$ to be
  given by $d(t)=0$ for $t\in[0,1]$ and $d(t)=2$ for $t\in[1,2]$.
  Finally, assume that, at time~$0$, store~$1$ is full while store~$2$
  is empty.  Then it is easy to see the only way in which the demand
  signal can be completely served for all $t\in[0,2]$ is to fully
  charge store~$2$ from store~$1$ during the time period $[0,1]$ and
  to then fully discharge both stores during the time period $[1,2]$.
  Hence in this case the need for cross-charging cannot be dispensed
  with.
\end{example}

Theorems~\ref{theorem:2} and \ref{theorem:3} have the following
corollary.

\begin{corollary}
  \label{corollary:3}
  % Suppose that cross-charging is allowed.  
  Suppose that the demand process $(d(t),t\in[0,T])$ is (weakly)
  increasing on $[0,T']$ and (weakly) decreasing on $[T',T]$ for some
  $0\le T'\le T$, and that all stores are full at time~$0$.  Then if
  $(d(t),t\in[0,T])$ can be served, it can be served without
  cross-charging, and by the use of the \ggddf\ policy.
\end{corollary}

\begin{proof}
  This is an application of Theorem~\ref{theorem:3} for the period
  $[0,T']$---including the result that any cross-charging may be
  eliminated in this period without reducing the volume of energy in
  each store at time~$T'$---followed by the use of
  Theorem~\ref{theorem:2} for the period $[T',T]$ (since clearly any
  increase in the volume of energy in each store at the time~$T'$ due
  to the elimination of cross-charging in $[0,T']$ means that we may
  continue to serve the given demand process on the interval $[T',T]$.
  Finally, since no cross-charging is necessary, it follows from
  Theorem~\ref{theorem:1} that the demand $(d(t),t\in[0,T])$ can be
  served by the use of the \ggddf\ policy.
\end{proof}

The following further corollary is an immediate application of
Theorem~\ref{theorem:1} and Corollary~\ref{corollary:3}.

\begin{corollary}
  \label{corollary:4}
  Under the conditions of Corollary~\ref{corollary:3}, and in a
  possibly stochastic environment, the minimisation of the expectation
  of the unserved energy demand~(\ref{eq:4}), or the minimisation of
  any quantile of the distribution of the unserved energy demand, is
  achieved by the use of the \ggddf\ policy.
\end{corollary}

An important application of the above result is to the frequently
occurring case where stores may be fully recharged overnight, and
there is a single period of shortfall during the day which is
\emph{unimodal} in the sense that it is monotonic increasing and then
decreasing as in the statement of Corollary~\ref{corollary:3}.  One
may assume that there is no surplus energy available for charging any
of the stores during this period.  Then, in a possibly stochastic
environment, the expectation of the unserved energy
demand~(\ref{eq:4}) is minimised by the use of the unique \ggddf\
policy.

Finally in this section, and for completeness, Theorem~\ref{theorem:4}
below gives a useful variation of Theorem~\ref{theorem:3}.  As should
be clear from the statement of Theorem~\ref{theorem:3} itself, the
condition of that theorem that ``all stores are full at time~$0$'' may
be relaxed subject to the additional restriction of some minimum level
on the demand function.  Theorem~\ref{theorem:4} makes this idea
precise.  It does of course also have corollaries analogous to
Corollaries~\ref{corollary:3} and~\ref{corollary:4} above.

\begin{theorem}
  \label{theorem:4}
  Consider any initial stored-energy configuration $(E_i(0), i\in\S)$.
  Let $\S'=\{i \in \S: E_i(0) < \overline{E}_i\}$ and define
  $u_0=\min_{i \in \S'}(E_i(0)/P_i)$, with $u_0=\infty$ if all stores
  in $\S$ are initially full.  Let $\S_1$ be the set of stores~$i$
  such that $\overline E_i/P_i \le u_0$.  (Note that from the
  definition of $u_0$ all the stores in the set~$\S_1$ are necessarily
  full.)  Define also $\S_2=\S\setminus\S_1$ and let
  $k=\sum_{i\in\S_2}P_i$.  Suppose that the demand function
  $(d(t),t\in[0,T])$ is (weakly) increasing, further satisfies
  $d(0)\ge k$, and may be served possibly with the use of
  cross-charging.  Then it may also be served without cross-charging.
\end{theorem}

\begin{proof}
  Consider a modified set of store capacities $(\hat E_i, i\in\S)$
  given by
  \begin{equation}
    \label{eq:22}
    \hat E_i =
    \begin{cases}
      \overline E_i, & \quad i \in \S_1,\\
      E_i(0) + T' P_i, & \quad i \in \S_2,
    \end{cases}
  \end{equation}
  where the constant~$T'=\max_{i\in\S_2}(\overline E_i - E_i(0))/P_i$.
  Observe that these modified store capacities are all at least as
  great as the original capacities.  Let the demand function
  $(d(t),t\in[0,T])$ be as in the statement of the theorem.  Extend
  the time interval $[0,T]$ to $[-T',T]$ and consider the demand
  function $(\hat d(t),t\in[-T',T])$ on this latter interval given by
  \begin{equation}
    \label{eq:23}
    \hat d(t) =
    \begin{cases}
      k, & \quad t \in [-T',0],\\
      d(t) & \quad t \in [0,T].
    \end{cases}
  \end{equation}
  If the modified set of stores are considered to be full at
  time~$-T'$ then their energy content is sufficient to serve the
  demand function $(\hat d(t),t\in[-T',T])$, again possibly with the
  use of cross-charging.  To see this, observe that the stores in the
  set~$\S_2$ may be utilised at their full rates to directly serve the
  demand $\hat d(t)=k$ on the interval $[-T',0]$ (as would be the case
  with the use of the \ggddf\ algorithm on that interval); at time~$0$
  the remaining energy in each store~$i$ is then the original energy
  content $E_i(0)$ of that store in the initial energy configuration
  of the theorem; the demand $\hat d(t)=d(t)$ on the remaining
  interval $[-T',T]$ may now be served as in the statement of the
  theorem (possibly with cross-charging, since the modified stores are
  all at least as large as the original stores).  Since the demand
  function $(\hat d(t),t\in[-T',T])$ is, by construction, increasing
  on $[-T',T]$, it follows from Theorem~\ref{theorem:3} that this
  demand function may also be served by the modified stores without
  cross-charging, and in particular by the use of the \ggddf\ algorithm.
  Since the residual energy content of these stores at time~$0$ is
  again simply the original energy content of the original stores at
  time~$0$, the conclusion of the theorem now follows.
\end{proof}

The need, in Theorem~\ref{theorem:4}, for some condition such as the
requirement that the demand function $(d(t),\,t\ge0)$ has some minimum
level at time~$0$ is again shown by the earlier
Example~\ref{example:cc2}.  The issue here is essentially the same as
in Theorem~\ref{theorem:3}: informally, for a weakly increasing demand
function, it is the combination of an initially sufficiently low level
of demand and spare capacity in the stores at time~$0$, which enables
cross-charging to assist in fully serving the demand function where,
in the absence of such cross-charging, this might not be possible.

\section{Charging and discharging}
\label{sec:charging-discharging}

We now allow that, for every time $t\ge0$, both the demand $d(t)$ and
the rates~$r_i(t)$, $i\in\S$, may be arbitrary (in particular may take
either sign) subject to the constraints given
by~\eqref{eq:1}--\eqref{eq:3} and the condition
\begin{equation}
  \label{eq:26}
      \min(d(t),0) \le \sum_{i\in\S:r_i(t)\ge0}r_i(t) +
    \frac{1}{\eta}\sum_{i\in\S:r_i(t)<0}r_i(t) \le \max(d(t),0),
  \qquad t \ge 0,
\end{equation}
where, throughout this section, we assume that all stores $i\in\S$
have the same round-trip efficiency $\eta_i=\eta$.  (For some
discussion of the case where the stores have different round-trip
efficiencies, see below and also Section~\ref{sec:conc}.)  Thus, when
the demand $d(t)\ge0$ the stores collectively serve energy (possibly
with cross-charging) to assist in meeting that demand and when
$d(t)<0$, corresponding to a surplus of energy external to the stores,
some of that surplus may be supplied to the stores (again possibly
with cross-charging).  Our objective continues to be to manage the
stores so as to minimise, over those times~$t$ such that $d(t)\ge0$,
the long-term unserved energy demand given by~\eqref{eq:4}, or, in a
stochastic environment, the expectation of this quantity.

For any continuous period of time $[T,T']$ over which $d(t)\le0$
(energy may be supplied to charge stores), there is now a theory which
is fully analogous to that developed in
Section~\ref{sec:pure-discharge-model} for discharging stores.  In
particular we may define the \emph{(residual) charge-duration} of any
store $i\in\S$ at any time~$t$ as $(\overline{E}_i - E_i(t))/P'_i$
(the time that would be required to fully charge the store if this was
done at the maximum possible rate). Further the \ggddf\ policy for
discharging-only which is optimal in the sense of
Theorem~\ref{theorem:1} is replaced by the analogous \emph{greedy
  greatest-charge-duration-first} (\ggcdf) policy, which, in the
absence of cross-charging, is similarly optimal for attempting to
accept as much charge as possible over the interval $[T,T']$.  (Note
that this would no longer be true in the absence of a common
round-trip efficiency for the stores.  While one might, in this case,
continue to formulate a result analogous to Theorem~\ref{theorem:1} by
redefining store capacities and input rates in terms of external
energy input rather than output, the results of the present section
would not then continue to hold.)

In many circumstances a reasonable \emph{heuristic} policy for the
management of the stores is given by the use of the \ggddf\ policy to
serve as much of the demand process as possible during periods when
that process is positive, and the use of the \ggcdf\ policy to charge
the stores as rapidly as possible during periods when the demand
process is negative.  This policy again has the attractive property,
discussed in Section~\ref{sec:pure-discharge-model}, of being
non-anticipatory, so that it continues to be fully implementable in a
stochastic environment.  As in Section~\ref{sec:disch-with-cross},
this policy may be expected to work particularly well when continuous
periods of storage discharge are separated by lengthy periods
providing ample time for recharging.  However, in general, it need not
always be \emph{optimal} for the minimisation of long-term unserved
energy.  The reason for this is that the \ggddf\ policy attempts to
equalise as quickly as possible the discharge-durations of the
stores, while the \ggcdf\ policy attempts to equalise as quickly as
possible the corresponding charge-duration times, and these two
objectives in general conflict with each other.  However, we do have
the following theorem.

\begin{theorem}
  \label{theorem:5}
  Suppose that a set of stores~$\S$ is such that $\overline E_i/P_i$
  is the same for all~$i\in\S$, that $P'_i=\alpha P_i$ for
  all~$i\in\S$ for some $\alpha>0$, and that the stores in~$\S$ have a
  common round-trip efficiency $\eta$.  Suppose also that $E_i(0)/P_i$
  is the same for all~$i\in\S$.
  % Suppose also that both charging and discharging are allowed.
  % Finally, suppose further that the (possibly stochastic) demand
  % signal $(d(t), t\in[0,T])$ may be both positive and negative.
  Then an optimal policy for the minimisation of the expectation of
  the unserved energy demand~(\ref{eq:4}) over any subsequent time
  period $[0,T]$ is given by the use of the \ggddf\ policy at those
  times~$t$ such that $d(t)$ is positive, and the use of the \ggcdf\
  policy at those times~$t$ such that $d(t)$ is negative.
\end{theorem}

\begin{proof}

  Under the conditions of the theorem, at any time~$t$, balance with
  respect to charging ($E_i(t)/P_i$ constant over $i\in\S$) is
  equivalent to balance with respect to discharging
  ($(\overline E_i - E_i(t))/P'_i$ constant over $i\in\S$).  The
  initial stored-energy configuration is balanced, and so, since the
  stores have a common round-trip efficiency~$\eta$, the
  \ggddf/\ggcdf\ policy maintains a balanced stored-energy
  configuration for all $t\ge0$.  It now follows from the results of
  Section~\ref{sec:pure-discharge-model}---see in particular
  Remark~\ref{remark:1}---that any demand process $(d(t),\,t\ge0)$
  which may be completely managed (i.e. completely served at times~$t$
  such that $d(t)\ge0$ and completely utilised for charging at
  times~$t$ such that $d(t)<0$) under some policy may similarly be
  completely managed by the use of the \ggddf/\ggcdf\
  policy.  The conclusion of the theorem now follows as for
  Corollary~\ref{corollary:2}.
\end{proof}

\begin{remark}\label{remark:2}
  The essence of Theorem~\ref{theorem:5} is that, under the conditions
  on device parameters given by the theorem, once stores become
  balanced with respect to their initial energy configurations, they
  remain so thereafter and may then be used with the same flexibility
  as a single large store.  It follows as in
  Section~\ref{sec:pure-discharge-model} that the \ggddf/\ggcdf\
  policy drives any initial energy configuration towards such a
  balanced state.  However, cross-charging may speed such convergence.
\end{remark}

\section{Concluding remarks}
\label{sec:conc}

The present paper has considered the optimal or near optimal
scheduling of heterogeneous storage resources for the ongoing
balancing of electricity supply and demand.  The results, which hold
in both deterministic and stochastic environments, are particularly
applicable when storage is used to cover periods of energy shortfall,
such as periods of daily peak demand, and may be completely or mostly
recharged between such periods.  However, in future years storage may
also be used to balance electricity systems over much longer
timescales, such as between summer and winter.  The simultaneous
existence of multiple timescales, together with the physical
characteristics of such storage types as are available, is likely to
lead to a very heterogeneous storage fleet, in which, in particular,
round-trip efficiencies vary considerably.  The optimal dimensioning
and control of such storage presents significant further research
challenges.

Another topic of ongoing research is to extend the storage dispatch
methodology beyond a scenario with a central dispatcher and develop,
for example, hierarchical schemes for aggregation and disaggregation
of energy units.

\section*{Acknowledgements}
\label{sec:acknowledgements}

The authors are grateful to the Isaac Newton Institute for
Mathematical Sciences in Cambridge and to National Grid plc for their
support for the programme in which much of the current work was
carried out.  They also thank Andy Philpott for some helpful
literature discussion.  Finally, the authors are especially grateful
to the reviewers for their most careful and thoughtful readings of the
paper and helpful comments and suggestions.

\bibliography{storage_refs_RSTA}
\bibliographystyle{plain}

\end{document}